\documentclass[12pt]{amsart}
\usepackage{amssymb,latexsym,tikz,hyperref,cleveref}
\usepackage[hmargin=1in,vmargin=1in]{geometry}

\newtheorem{thm}{Theorem}[section]
\newtheorem{prop}[thm]{Proposition}
\newtheorem{cor}[thm]{Corollary}
\newtheorem{lem}[thm]{Lemma}
\newtheorem{conj}[thm]{Conjecture}
\newtheorem{exa}[thm]{Example}

\theoremstyle{definition}
\newtheorem{defn}[thm]{Definition}

\newcommand{\da}{\downarrow}
\newcommand{\rhoh}{\hat{\rho}}
\newcommand{\chih}{\hat{\chi}}
\newcommand{\blr}[2]{\filldraw[fill=black, draw=white, very thick] #1 rectangle #2;}
\newcommand{\yer}[2]{\filldraw[fill=yellow, draw=white, very thick] #1 rectangle #2;}
\DeclareMathOperator{\PL}{PL}
\DeclareMathOperator{\B}{B}

\newcommand{\ben}{\begin{enumerate}}
\newcommand{\een}{\end{enumerate}}
\newcommand{\ble}{\begin{lem}}
\newcommand{\ele}{\end{lem}}
\newcommand{\bth}{\begin{thm}}
\newcommand{\bpr}{\begin{prop}}
\newcommand{\epr}{\end{prop}}
\newcommand{\bco}{\begin{cor}}
\newcommand{\eco}{\end{cor}}
\newcommand{\bcon}{\begin{conj}}
\newcommand{\econ}{\end{conj}}
\newcommand{\bde}{\begin{defn}}
\newcommand{\ede}{\end{defn}}
\newcommand{\bex}{\begin{exa}}
\newcommand{\eex}{\end{exa}}
\newcommand{\barr}{\begin{array}}
\newcommand{\earr}{\end{array}}
\newcommand{\btab}{\begin{tabular}}
\newcommand{\etab}{\end{tabular}}
\newcommand{\beq}{\begin{equation}}
\newcommand{\eeq}{\end{equation}}
\newcommand{\bea}{\begin{eqnarray*}}
\newcommand{\eea}{\end{eqnarray*}}
\newcommand{\bal}{\begin{align*}}
\newcommand{\bce}{\begin{center}}
\newcommand{\ece}{\end{center}}
\newcommand{\bpi}{\begin{picture}}
\newcommand{\epi}{\end{picture}}
\newcommand{\bpp}{\begin{picture}}
\newcommand{\epp}{\end{picture}}
\newcommand{\bfi}{\begin{figure} \begin{center}}
\newcommand{\efi}{\end{center} \end{figure}}
\newcommand{\bprf}{\begin{proof}}
\newcommand{\eprf}{\end{proof}\medskip}

\newcommand{\bsl}{\begin{slide}{}}
\newcommand{\esl}{\end{slide}}
\newcommand{\bfr}{\begin{frame}}
\newcommand{\efr}{\end{frame}}

\newcommand{\hqed}{\hfill \qed}

\newcommand{\eqqed}[1]{$\rule{1ex}{0ex}\hfill{\dil#1}\hfill\qed$}

\newcommand{\hs}[1]{\hspace{#1}}
\newcommand{\hso}[1]{\hspace{-1pt}}
\newcommand{\vs}[1]{\vspace{#1}}

\newcommand{\emp}{\emptyset}

\newcommand{\sbe}{\subseteq}

\newcommand{\setm}{\setminus}

\newcommand{\lec}{\lessdot}
\newcommand{\grc}{\gtrdot}

\newcommand{\zh}{\hat{0}}
\newcommand{\oh}{\hat{1}}
\newcommand{\Ph}{\hat{P}}

\newcommand{\case}[4]{\left\{\barr{ll}#1&\mbox{#2}\\#3&\mbox{#4}\earr\right.}

\def\<{\langle}
\def\>{\rangle}

\newcommand{\ra}{\rightarrow}

\newcommand{\al}{\alpha}
\newcommand{\be}{\beta}

\newcommand{\de}{\delta}

\newcommand{\si}{\sigma}

\newcommand{\bbF}{{\mathbb F}}

\newcommand{\bbZ}{{\mathbb Z}}
\newcommand{\cA}{{\mathcal A}}

\newcommand{\cG}{{\mathcal G}}

\newcommand{\cI}{{\mathcal I}}

\newcommand{\cL}{{\mathcal L}}

\newcommand{\cM}{{\mathcal M}}

\newcommand{\cO}{{\mathcal O}}

\newcommand{\cS}{{\mathcal S}}

\def\multiset#1#2{\ensuremath{\left(\kern-.3em\left(\genfrac{}{}{0pt}{}{#1}{#2}\right)\kern-.3em\right)}}

\DeclareMathOperator{\lcm}{lcm}

\DeclareMathOperator{\st}{st}

\newcommand{\dil}{\displaystyle}

\begin{document}
\pagestyle{plain}

\title{Rowmotion on rooted trees
}
\author{Pranjal Dangwal}
\address{ Department of Mathematics, Michigan State University,
 East Lansing, MI 48824, USA}
\email{dangwalp@msu.edu}
\author{Jamie Kimble}
\address{ Department of Mathematics, Michigan State University,
 East Lansing, MI 48824, USA}
\email{kimblej2@msu.edu}
\author{Jinting Liang}
\address{ Department of Mathematics, Michigan State University,
 East Lansing, MI 48824, USA}
\email{liangj26@msu.edu}
\author{Jianzhi Lou}
\address{ Department of Mathematics, Michigan State University,
 East Lansing, MI 48824, USA}
\email{loujianz@msu.edu}
\author{Bruce E. Sagan}
\address{ Department of Mathematics, Michigan State University,
 East Lansing, MI 48824, USA}
\email{bsagan@msu.edu}
\author{Zach Stewart}
\address{ Department of Mathematics, Michigan State University,
 East Lansing, MI 48824, USA}
\email{stewa719@msu.edu}

\date{\today}


 
\maketitle

\begin{abstract}
A rooted tree $T$ is a poset whose Hasse diagram is a graph-theoretic tree having  a unique minimal element.  We study rowmotion on antichains and lower order ideals of $T$.  Recently Elizalde, Roby, Plante and Sagan considered rowmotion on fences which are posets whose Hasse diagram is a path (but permitting any number of minimal elements).  They showed that in this case, the orbits could be described in terms of tilings of a cylinder.  They also defined a new notion called homometry which means that a statistic takes a constant value on all orbits of the same size.  This is a weaker condition than the well-studied concept of homomesy which requires a constant value for the average of the statistic over all orbits.  Rowmotion on fences is often homometric for certain statistics, but not homomesic.  We introduce a tiling model for rowmotion on rooted trees.  We use it to study various specific types of trees and show that they exhibit homometry, although not homomesy, for certain statistics.
\end{abstract}

\tableofcontents


\section{Introduction}

Let  $S$ be a set with $\#S$ finite where the hash symbol denotes cardinality.  A {\em statistic} on $S$ is a function $\st:S\ra\bbZ$ where $\bbZ$ is the integers.  We extend $\st$ to subsets $R\sbe S$ by letting
$$
\st R = \sum_{r\in R} \st r.
$$
Now suppose that $G$ is a finite group acting on $S$.  Statistic $\st$ is said to be
{\em homomesic} if, for any orbit $\cO$ of $G$, we have
$$
\frac{\st \cO}{\#\cO} = c
$$
for some constant $c$.  To be more specific, we say in this case that this statistic is {\em $c$-mesic}.  Homomesy is a well-studied property; see the survey articles of Roby~\cite{rob:dac} or Striker~\cite{str:rgt}.
Recently Elizalde, Roby, Plante, and Sagan~\cite{EPRS:rf} introduced a weaker notion which is exhibited by certain actions and statistics.  
We say that a statistic $\st$ is {\em homometric} if
for any two orbits $\cO_1$ and $\cO_2$ of the same cardinality we have $\st\cO_1 = \st\cO_2$.  We will see numerous examples of statistics which are homometric but not homomesic in the present work.

Now consider a finite partially ordered set, often abbreviated to {\em poset}, $(P,\le)$.  An {\em antichain} of $P$ is a $A\sbe P$ such that no two elements of $A$ are comparable. We denote the set of all antichains as 
$$
\cA(P) = \{A\sbe P \mid \text{$A$ is an antichain}\}.
$$
A {\em lower order ideal of $P$} is $L\sbe P$ such that if $y\in L$ and $x\le y$ then $x\in L$.  We will use the notation
$$
\cL(P) = \{L\sbe P \mid \text{$L$ is a lower order ideal}\}.
$$
The lower order ideal {\em generated} by any $Q\sbe P$ is
$$
Q\da\  = \{x\in P \mid \text{$x\le y$ for some $y\in Q$}\}.
$$
We also let $\min Q$ and $\max Q$ be the sets of minimal and maximal elements of $Q$, respectively.
We now define {\em rowmotion on antichains} to be the action generated by $\rho:\cA(P)\ra\cA(P)$ where
$$
\rho(A)  =\min\{x\not\in (A\da)\}.
$$
Similarly, {\em rowmotion on ideals} has generator $\rhoh:\cL(P)\ra\cL(P)$  with
$$
\rhoh(L) =  \rho(\max L)\da.
$$
We will usually use a hat to distinguish a notation on ideals from the corresponding one on antichains.  More information about rowmotion can be found in the aforementioned survey articles.

The paper of Elizalde et al.\ was devoted to the study of rowmotion on fences.  A fence is a poset whose Hasse diagram is a path.  They showed that the antichain orbits can be modeled using certain tilings of a cylinder.  This tool permitted them to prove a number of homometries which were not homomesies.  In the present work we will consider rowmotion on rooted trees.    A poset $T$ is a {\em rooted tree} if its Hasse diagram is a tree in the graph theory sense of the term, and it has a unique minimal element called the {\em root} and denoted $\zh$.  Note that these posets are more general than fences in that the tree need not be a path, but also more restricted in that fences can have any number of minimal elements.  We will assume all our trees our rooted. 

The rest of this paper is structured as follows.  In the next section we will show that rowmotion on antichains of a rooted tree can also be viewed in terms of certain cylindrical tilings.  The following three sections will apply this tiling model to three different families of trees: stars, trees with three leaves, and finally combs and zippers.  We end with a section with comments and open questions.


\section{Tilings}
\label{t}

\begin{figure}
 \begin{tikzpicture}
\fill(2,0) circle(.1);
\fill(2,1) circle(.1);
\fill(1,2) circle(.1);
\fill(3,2) circle(.1);
\fill(0,3) circle(.1);
\fill(1,3) circle(.1);
\fill(4,3) circle(.1);
\fill(0,4) circle(.1);
\fill(1,4) circle(.1);
\fill(3,4) circle(.1);
\fill(4,4) circle(.1);
\fill(5,4) circle(.1);
\fill(1,5) circle(.1);
\fill(3,5) circle(.1);
\fill(5,5) circle(.1);
\draw (2,0)--(2,1)--(0,3)--(0,4) (1,2)--(1,5) (2,1)--(5,4)--(5,5) 
(4,4)--(4,3)--(3,4)--(3,5);
\draw(0,4.5) node{$1$};
\draw(1,5.5) node{$2$};
\draw(3,5.5) node{$3$};
\draw(4,4.5) node{$4$};
\draw(5,5.5) node{$5$};
\draw(3.3,1.7) node{$x$};
\draw(4.3,2.7) node{$y$};
\draw(2,-1) node{$T$};
 \end{tikzpicture}
\hs{70pt}   
\begin{tikzpicture}
\fill(2,0) circle(.1);
\fill(2,1) circle(.1);
\fill(1,2) circle(.1);
\fill(3,2) circle(.1);
\fill(0,3) circle(.1);
\fill(1,3) circle(.1);
\fill(4,3) circle(.1);
\fill(0,4) circle(.1);
\fill(1,4) circle(.1);
\fill(3,4) circle(.1);
\fill(4,4) circle(.1);
\fill(5,4) circle(.1);
\fill(1,5) circle(.1);
\fill(3,5) circle(.1);
\fill(5,5) circle(.1);
\draw (2,0)--(2,1) (3,2)--(4,3) (0,3)--(0,4) (1,3)--(1,5)
(3,4)--(3,5) (5,4)--(5,5);
\draw(-.5,4.5) node{$([1,1],2)$};
\draw(1,5.5) node{$([2,2],3)$};
\draw(3,5.5) node{$([3,3],2)$};
\draw(4,4.5) node{$([4,4],1)$};
\draw(5,5.5) node{$([5,5],2)$};
\draw(4.3,1.7) node{$x=x_{[3,5],2}$};
\draw(5,2.7) node{$y=x_{[3,5],1}$};
\draw(1,.5) node{$([1,5],2)$};
\draw(0,2) node{$([1,2],1)$};
\draw(4.5,2.2) node{$([3,5],2)$};
\draw(2,-1) node{$\cI(T)$};
 \end{tikzpicture}
    \caption{The intervals, branches, and $\be$-values of a tree $T$}
    \label{int:fig}
\end{figure}
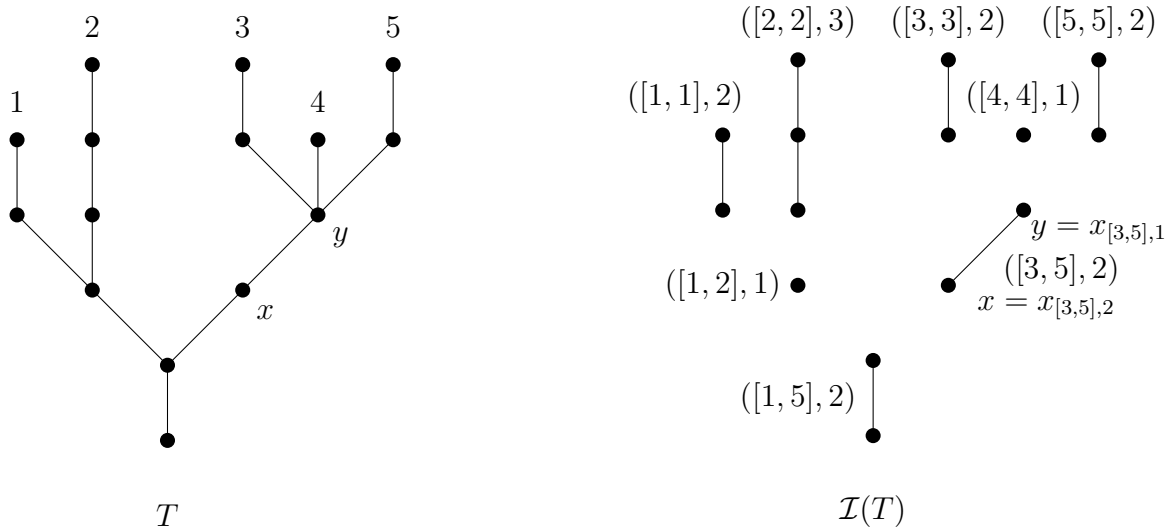

We will show that rowmotion orbits on antichains can be more easily viewed as certain tilings of a cylinder.  Given a rooted tree $T$ we will fix an embedding of the Hasse diagram of $T$ in the plane and label its leaves (maximal elements) as
$1,2,\ldots,n$ from left to right.  See the tree on the left of Figure~\ref{int:fig} for an example where $n=5$.  

For nonnegative integers $m,n$ we use interval notation
$$
[m,n]=\{m,m+1,\ldots,n\}
$$
and abbreviate $[n]=[1,n]$.  Associate with each vertex $x$ in $T$ the set of all labels of leaves $z$ such that $z\ge x$.   Note that by our choice of labeling, this set will be an interval $I$.  And the set of all $x$ with interval $I$ form a path called the {\em branch} corresponding to $I$ and denoted $B_I$. On the right in Figure~\ref{int:fig}, $T$ has been decomposed into branches with each labeled by a pair where the first component is the interval $I$ of $B_I$.  For example, nodes $x$ and $y$ are exactly the ones  below all three leaves $3,4,5$.  So their associated interval is $I=[3,5]$ and $B_{[3,5]}=\{x,y\}$.  
We will also label the vertices on the branch for $I$ as $x_{I,1},x_{I,2},\ldots$ starting with the maximal element and working down. 
Returning to our example,
$y=x_{[3,5],1}$ and $x=x_{[3,5],2}$.
Note the following two simple but important properties of this family of intervals.
\begin{enumerate}
    \item[(I1)]  The singleton intervals $[i,i]$ are in this family for all $i\in[n]$.
    \item[(I2)]  The family is {\em nested} in the sense that if $I,J$ are in the family with $\#I\le \#J$ then either $I\sbe J$ or
    $I\cap J=\emp$.
\end{enumerate}

Given an interval $I$, let 
$$
\be_I=\be_I(T)=\#B_I.
$$
Returning to our usual example, for $I=[3,5]$ we saw that $B_{[3,5]}=\{x,y\}$ which implies $\be_{[3,5]}=2$.   
A crucial tool in defining the tilings will be the set
$$
\cI(T) = \{(I,\be_I) \mid \text{$I$ is the interval of some branch of nodes in $T$}\}.
$$
On the right in Figure~\ref{int:fig}, the elements of $\cI(T)$ are displayed next to their corresponding branches.  We will abuse notation and write $I\in\cI(T)$ to mean that $(I,\be_I)\in\cI(T)$.

We will need to consider partitions of intervals.
A {\em partition} of an interval $I$ is a collection of nonempty subintervals $I_1,\ldots,I_k$ whose disjoint union is $I$.  We say that another partition $J_1,\ldots,J_l$ of $I$ is a {\em refinement} of the first if for every $J_j$ there is an $I_i$ with $J_j\sbe I_i$.
The refinement is {\em proper} if the two collections of subintervals are not the same.
Refinement is a partial order on partitions.  If all the intervals of the partition come from $\cI(T)$ then it is called an  {\em $\cI(T)$-partition}.

\begin{figure}
 \begin{tikzpicture}
\fill(2,0) circle(.1);
\fill(2,1) circle(.1);
\fill(1,2) circle(.1);
\fill(3,2) circle(.1);
\fill(0,3) circle(.1);
\fill(1,3) circle(.1);
\fill(4,3) circle(.1);
\fill(0,4) circle(.1);
\fill(1,4) circle(.1);
\fill(3,4) circle(.1);
\fill(4,4) circle(.1);
\fill(5,4) circle(.1);
\fill(1,5) circle(.1);
\fill(3,5) circle(.1);
\fill(5,5) circle(.1);
\draw (2,0)--(2,1)--(0,3)--(0,4) (1,2)--(1,5) (2,1)--(5,4)--(5,5) 
(4,4)--(4,3)--(3,4)--(3,5);
\draw(0,4.5) node{$1$};
\draw(1,5.5) node{$2$};
\draw(3,5.5) node{$3$};
\draw(4,4.5) node{$4$};
\draw(5,5.5) node{$5$};
\draw(-.5,4) node{$u$};
\draw(1.5,3) node{$v$};
\draw(3.3,1.7) node{$x$};
\draw(4.3,2.7) node{$y$};
\draw(2,-1) node{$T$};
 \end{tikzpicture}
\hs{20pt}
\begin{tikzpicture}
\draw(-.5,5.5) node{$1$};
\draw(-.5,4.5) node{$2$};
\draw(-.5,3.5) node{$3$};
\draw(-.5,2.5) node{$4$};
\draw(-.5,1.5) node{$5$};
\blr{(0,1)}{(1,4)}
\yer{(0,4)}{(1,5)}
\blr{(0,5)}{(1,6)}
\draw(.5,0) node{$\{u,x\}$};
\draw(2,3.5) node{$\stackrel{\rho}{\mapsto}$};
\begin{scope}[shift={(3,0)}]
\blr{(0,1)}{(1,4)}
\blr{(0,4)}{(1,5)}
\yer{(0,5)}{(1,6)}
\draw(.5,0) node{$\{v,y\}$};
\draw(2,3.5) node{$=$};
\end{scope}
\begin{scope}[shift={(6,0)}]
\blr{(0,1)}{(2,4)}
\yer{(0,4)}{(1,5)}
\blr{(0,5)}{(1,6)}
\blr{(1,4)}{(2,5)}
\yer{(1,5)}{(2,6)}
\end{scope}
\end{tikzpicture}
 
    \caption{Rowmotion on antichains in terms of tilings}
    \label{tile:fig}
\end{figure}
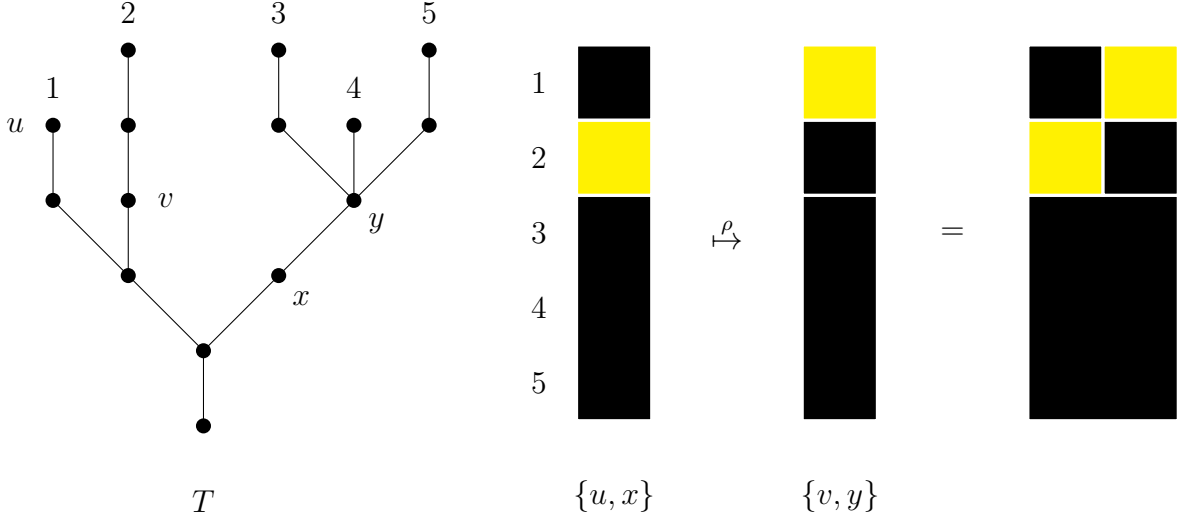

We now describe the procedure to produce a tiling $\tau$ from an orbit $\cO$ of rowmotion on antichains of a rooted tree $T$.  Consider a column of $n$ boxes where the $i$th box corresponds to the leaf labeled $i$ in the embedding of $T$.  The first column in Figure~\ref{tile:fig} is so labeled.  Given an antichain $A$, we take each $x\in A$ and consider the interval $I$ of its branch.  The boxes labeled by the elements of $I$ are then covered by a black tile.  All other boxes are covered by a single yellow tile.  Note that these boxes are exactly the ones in rows $i$ such that there is no element of $A$ below the leaf labeled $i$.  Returning to Figure~\ref{tile:fig}, consider the antichain $A=\{u,x\}$ and the leftmost column of tiles.  Since $u$ has interval $[1,1]$, the box in row $1$ gets a black tile.  Similarly, $x$'s interval is $[3,5]$ so the rows for this interval also receive a black tile.  The remaining square in row $2$ then receives a yellow tile.  The reader should now not find it hard to verify that $\rho(A)=\{v,y\}$ and that this antichain corresponds to the second column in the figure.  We now paste the columns for all antichains in the orbit $\cO$ together in the same order that they appear in the orbit to get a tiling $\tau=\tau(\cO)$ of a cylinder.  Note that when pasting, if there are two consecutive columns with black tiles coming from the same interval $I$ then these tiles are combined into one.  Returning to our perennial example, the two tiles for $I=[3,5]$ become one tile as seen in the final diagram.  And if there were more elements on the branch corresponding to $I$, they would fatten the tile further.  Three tilings corresponding to full orbits are shown in Figure~\ref{S(3,3,2)}.  The vertical sides of these rectangles are identified to make them into cylinders.  Also note that, in the middle tiling, a black tile in the second row stretches over this boundary as indicated by having it protrude beyond the sides of the rectangle. 

We wish to characterize the possible $\tau(\cO)$.  In the definition below, an $I\times b$ tile is a tile which covers the rows indexed by $I$ and $b$ columns.  Also, the maximal partitions used are maximal with respect to the refinement order.  They exist because property (I1) implies that any interval $I$ has a partition using intervals in $\cI(T)$ since all singletons are intervals.  And property (I2) guarantees that among all such partitions of $I$ there is a maximal one.

\begin{defn}
\label{tau:def}
Given a rooted tree $T$, an {\em $\cI(T)$-tiling} is a tiling of a cylinder using 
$I\times \be_I$ black tiles  and $I\times 1$ yellow tiles if $\#I=1$, satisfying the following two properties.
\begin{enumerate}
    \item[(t1)]
    An $I\times \be_I$ black tile is followed by a yellow tile if $\#I=1$, or by black tiles corresponding to the intervals in a maximal proper $\cI(T)$-partition of $I$ if $\#I\ge2$.
    \item[(t2)] If $J$ is a maximal interval of yellow tiles in a column, then they are followed by black tiles corresponding to the intervals in a maximal $\cI(T)$-partition of $J$.
\end{enumerate}
\end{defn}

\begin{thm}
Given a rooted tree, $T$, the map $\cO\mapsto \tau(\cO)$ is a bijection between the antichain rowmotion orbits of $T$ and the possible $\cI(T)$-tilings.
\end{thm}
\begin{proof}
We must first show that this map is well defined in that $\tau=\tau(\cO)$ has tiles satisfying (t1) and (t2) and of the correct shape.   We will do this by studying how rowmotion affects various elements of $T$. 

Consider $A\in\cO$ and any $x\in A$ which is not maximal in its branch and let $I$ be the associated interval.
Then there is a unique element $y$ which covers $x$ and it is in the same branch.
Furthermore $y\in\rho(A)$.  Since $x$ and $y$ correspond to the same interval $I$, it follows that the tile covering those rows in the column for $A$
 extends into the column for $\rho(A)$.  By induction, this tile extends into a column for an antichain containing the maximal element on the branch.
 
Now suppose that $x\in A$ is maximal in its branch.  If $\#I=1$ then $v$ is maximal in $T$.  So in $\rho(A)$ the branch will be empty and the algorithm will place a yellow tile in the corresponding row and column.  This proves the first case in (t1).
On the other hand, if $\#I\ge2$ then $x$ is covered by at least two elements 
$y_1,\ldots, y_k$.  So the column for $\rho(A)$ will contain tiles in the corresponding intervals $I_1,\ldots,I_k$ which is a proper $\cI(T)$-partition of $I$ since $k\ge2$.  And it is maximal since if there is some $J\in\cI(T)$ containing two or more of the $I_i$ then there would have to be at least one element between $x$ and the corresponding $y_i$'s.  This completes the proof of (t1).

For (t2), we will assume for simplicity that $1,n\not\in J$ where $n$ is the number of leaves of $T$.  The cases when $J$ contains one or both of these special values is similar.  Say $J=[m,n]$.  Then by our assumption, there are black tiles covering rows $m-1$ and $n+1$ in the column for $J$.  Let $x$ and $y$ be the corresponding elements of $A$.  Removing the $\zh$--$x$ and $\zh$--$y$ paths from $T$ breaks the lower order ideal generated by the leaves in $J$ into rooted subtrees.   Let $z_1,\ldots, z_k$ be their roots with corresponding intervals $I_1,\ldots,I_k$.  Then  $\rho(A)$  contains these $z_i$ and so its column contains tiles for the intervals $I_i$ which form a partition of $J$.  Maximality is obtained by the same argument as in the previous paragraph.

To complete showing that $\tau$ is well defined, we must check the shape of the tiles.  Yellow tiles are of the correct shape by definition of the algorithm.  As far as the black tiles, they cover rows indexed by intervals in $\cI(T)$ by definition.  So it suffices to show that a tile in the rows indexed by $I$ has the correct length.  From the previous two paragraphs we see that the tiles in the partitions following the maximal element of a black tile or following an interval of yellow tiles all begin with the minimal elements of their respective branches.  And by the second paragraph, such a tile will extend to the maximal element on its branch.  So the tile will have length $\be_I$, the length of the branch.

To show that this map is a bijection, we construct its inverse.  So given an $\cI(T)$-tiling $\tau$, we must construct a corresponding orbit $\cO$.  For each column of $\tau$ we form an antichain $A$ as follows.  For each interval $I$ covered by a black tile, suppose the give column is the $i$th in that tile.  Then add the $i$th smallest element on the branch for $I$ to $A$.  
Now arrange the antichains in the same order as the columns of the tiling to get an orbit.
The demonstration that this map is well defined is similar to the one just given.  And the two functions are inverses since the algorithms described are step-by-step reversals.  This completes the proof.
\end{proof}

We will often call the tiles of shape $I\times\be_I$ simply {\em $I$-tiles}.
As a first application of the tiling model, we will use it to compute various statistics on rowmotion orbits.  It will also give us a simple proof of our first homomesy.  Given $x\in T$ we have the statistic on antichains $A\in\cA(T)$ given by
$$
\chi_x(A) = \case{1}{if $x\in A$,}{0}{if $x\not\in A$.}
$$
If we want to count the size of antichains we use the statistic
$$
\chi(A) = \sum_{x\in T}\chi_x(A) = \#A.
$$
The corresponding statistics for ideals are denoted
$\chih_x$ and $\chih$.  Given a $\cI(T)$-tiling $\tau$ we will use the notation
$$
m_I = m_I(\tau) = \text{number of $I$-tiles in $\tau$.}
$$
\begin{cor}
\label{I(T)}
Let $T$ be a rooted tree and $\tau$ be a $\cI(T)$-tiling corresponding to a rowmotion orbit $\cO$ on $T$.  The following hold.
\begin{enumerate}
    \item[(a)]  If $x\in T$ has interval $I$ then
    $$
    \chi_x(\cO) = m_I.
    $$
    \item[(b)] We have
    $$
    \chi(\cO) = \sum_{I\in\cI(T)} \be_I m_I.
    $$
    \item[(c)] If $x=x_{I,j}$ then
    $$
    \chih_x(\cO) = j\cdot m_I + c_I
    $$
    where $c_I$ is the number of columns of $\tau$ intersecting a $J$-tile for $J\subset I$.
    \item[(d)]  We have
    $$
    \chih(\cO) = \sum_{I\in\cI(T)}
    \left[\binom{\be_I+1}{2} m_I + \be_I c_I \right]
    $$
    \item[(e)]  If $x,y$ are in the same branch then $\chi_x-\chi_y$ is $0$-mesic.
\end{enumerate}
\end{cor}
\begin{proof}
(a)  This follows from the fact that $v$ is represented by a single column in each $I$-tile of $\tau$.

\medskip

(b)  Since $I$-tiles have length $\be_I=\#B_I$ we get by summing (a)
\begin{align*}
\chi(\cO) &= \sum_{x\in T}     \chi_x(\cO)\\
&=\sum_{I\in\cI(T)} \sum_{x\in B_I} m_I\\
&=\sum_{I\in\cI(T)} \be_I m_I.
\end{align*}

\medskip

(c)  For a lower order ideal $L$ we have that $x\in L$ if and only if $x\le y$ for some $y\in A$ where  $A=\max L$.  Note also that if $y$ has interval $J$ then $y\ge x$ implies $J\subseteq I$.  
If $J=I$ then there are $j$ choices for $y$ and so $j\cdot m_I$ counts the total number of columns containing such an element.  And  $c_I$ accounts for the columns intersection some $J$-tile where 
$J\subset I$.

\medskip

(d)  This result follows from (c) in much the same way that (b) followed from (a).  So the proof is left to the reader.

\medskip

(e)  Let the common branch be $B_I$.  Using (a) one last time we get
$$
\chi_x(\cO)-\chi_y(\cO) = m_I - m_I = 0
$$
which implies the homomesy.
\end{proof}

We end this section with a recursive formula for the number of antichains in a rooted tree $T$ which will be useful in the sequel.  
We use $T\setm\{\zh\}$ for the forest of rooted trees obtained by removing $\zh$ from $T$.
\begin{lem}
\label{num:A}
Let $T$ be a rooted tree.  If $\#T=1$ then $\#\cA(T)=2$.  If $\#T\ge2$ then let $T_1,\ldots,T_k$ be the rooted tree components of $T\setm\{\zh\}$.  In this case
$$
\#\cA(T) = 1 + \prod_{i=1}^k \#\cA(T_i).
$$
\end{lem}
\begin{proof}
If $\#T=1$ then $T$ has antichains $\emp$ and $\{\zh\}$.
When $\#T\ge2$, let $A$ be an antichain of $T$.  Either $A=\{\zh\}$, corresponding to the $1$ is the sum, or $A\sbe \uplus_i T_i$.  In the latter case the restriction $A_i$ of $A$ to $T_i$ is an antichain and the product counts the possible $A_i$.
\end{proof}

\section{Stars}
\label{s}

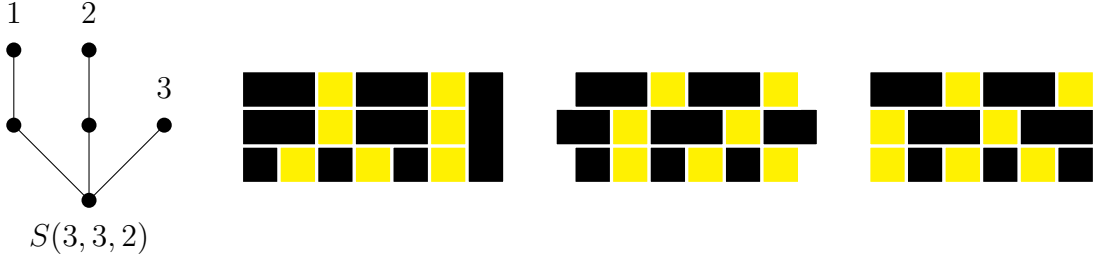
\begin{figure}
 \begin{tikzpicture}
\fill(1,0) circle(.1);
\fill(0,1) circle(.1);
\fill(1,1) circle(.1);
\fill(2,1) circle(.1);
\fill(0,2) circle(.1);
\fill(1,2) circle(.1);
\draw (0,2)--(0,1)--(1,0)--(1,2) (1,0)--(2,1);
\draw(0,2.5) node{$1$};
\draw(1,2.5) node{$2$};
\draw(2,1.5) node{$3$};
\draw(1,-.5) node{$S(3,3,2)$};
 \end{tikzpicture}
\hs{10pt}
\raisebox{30pt}{
 \begin{tikzpicture}[scale=.5]
 \blr{(0,0)}{(1,1)}
 \yer{(1,0)}{(2,1)}
  \blr{(2,0)}{(3,1)}
 \yer{(3,0)}{(4,1)}
  \blr{(4,0)}{(5,1)}
 \yer{(5,0)}{(6,1)}
 \blr{(0,1)}{(2,2)}
 \yer{(2,1)}{(3,2)}
  \blr{(3,1)}{(5,2)}
  \yer{(5,1)}{(6,2)}
 \blr{(0,2)}{(2,3)}
 \yer{(2,2)}{(3,3)}
  \blr{(3,2)}{(5,3)}
  \yer{(5,2)}{(6,3)}
 \blr{(6,0)}{(7,3)}
 \end{tikzpicture}
 \hs{10pt}
 \begin{tikzpicture}[scale=.5]
 \blr{(0,0)}{(1,1)}
 \yer{(1,0)}{(2,1)}
  \blr{(2,0)}{(3,1)}
 \yer{(3,0)}{(4,1)}
  \blr{(4,0)}{(5,1)}
 \yer{(5,0)}{(6,1)}
 \blr{(-.5,1)}{(1,2)}
 \yer{(1,1)}{(2,2)}
  \blr{(2,1)}{(4,2)}
  \yer{(4,1)}{(5,2)}
 \blr{(5,1)}{(6.5,2)}  
 \blr{(0,2)}{(2,3)}
 \yer{(2,2)}{(3,3)}
  \blr{(3,2)}{(5,3)}
  \yer{(5,2)}{(6,3)}
 \end{tikzpicture}
 \hs{10pt}
 \begin{tikzpicture}[scale=.5]
 \yer{(0,0)}{(1,1)}
 \blr{(1,0)}{(2,1)}
  \yer{(2,0)}{(3,1)}
 \blr{(3,0)}{(4,1)}
  \yer{(4,0)}{(5,1)}
 \blr{(5,0)}{(6,1)}
\yer{(0,1)}{(1,2)}
 \blr{(1,1)}{(3,2)}
 \yer{(3,1)}{(4,2)}
  \blr{(4,1)}{(6,2)}
 \blr{(0,2)}{(2,3)}
 \yer{(2,2)}{(3,3)}
  \blr{(3,2)}{(5,3)}
  \yer{(5,2)}{(6,3)}
 \end{tikzpicture}
 }
    \caption{The star $S(3,3,2)$ and its tilings}
    \label{S(3,3,2)}
\end{figure}

A {\em star}, $S$, is a rooted tree with $n$ leaves and
$$
\cI(S) = \{([1,1],\be_1),\ \ldots,\ ([n,n],\be_n),\ ([n],1)\}
$$
where we are using the abbreviation $\be_i=\be_{[i,i]}$.  
We will use the same abbreviation for other notation involving a subscript $[i,i]$, for example $x_{i,j} = x_{[i,i],j}$.
So $S$ is the result of taking $n$ chains of length $\be_1,\ldots,\be_n$ and identifying their minimal elements.  Note that all tiles in a corresponding tiling will only cover one row, except for the tile corresponding to $\zh$.  We denote this star by $S(\al_1,\ldots,\al_n)$ where $\al_i=\be_i+1$ for $i\in[n]$.  The reason for this change of variables is because it will make our results easier to state since $\al_i$ is the length of a black tile followed by a yellow tile in row $i$.
The star $S(3,3,2)$ and its tilings are found in Figure~\ref{S(3,3,2)}.
Given an orbit $\cO$ we will use the notation
$$
\de=\case{1}{if $\zh\in\cO$,}{0}{if $\zh\not\in\cO$.}
$$
\begin{thm}
\label{star:thm}
Consider the star $S=S(\al_1,\ldots,\al_n)$ and an orbit $\cO$ of rowmotion on $S$.  Let $l=\lcm(\al_1,\ldots,\al_n)$.  
\begin{enumerate}
\item[(a)] We have 
\[
\#\cO=l+\de
\]
and the number of orbits is $\al_1\cdots\al_n/l$.
    \item[(b)] For any $x\in S$,
    $$\chi_x(\cO)=\begin{cases} 
      l/\alpha_i & \text{if $x\in B_i$},\\
     \delta & \text{if $x=\hat{0}$}.
   \end{cases}$$
    \item[(c)] We have
    $$
       \chi(\cO)=\delta+\sum_{i=1}^n \frac{l}{\alpha_i}(\alpha_i-1).
    $$
    Thus $\chi$ is homometric but not homomesic.
    \item[(d)] For any $x\in S$
   $$
   \chih_x(\cO)=\case{jl/\alpha_i}{if $x=x_{i,j}$}{l}{if $x=\hat{0}$.}
     $$
     \item[(e)]  We have
     $$
        \hat{\chi}(\cO)= l+\sum_{i=1}^n \frac{l}{\alpha_i}\binom{\alpha_i}{2}.
     $$
     Thus $\chih$ is homometric but not homomesic.
\end{enumerate}
\end{thm}
\begin{proof}
(a)  Consider the tiling $\tau=\tau(\cO)$.  For all $i\in[n]$ the corresponding interval $I=[i,i]$ has $\#I=1$.  So, by condition (t1) in Definition~\ref{tau:def}, each black tile in that row is followed by a yellow tile.  And this pair of tiles has length $\be_i+1=\al_i$.

Now consider the case when $\zh\not\in\cO$.  So no tile spans more than one row.  Now the previous paragraph and (t2) imply
that the black and yellow tiles alternate in row $i$.  So the length of that row is divisible by $\al_i$.  Since this is true for all $i$ we must have that $l$ divides $\#\cO$.  But since $l$ is the least common multiple, a given column will recur after $l$ steps.  So we must have $\#\cO=l$.  When $\zh\in\cO$ then the same reasoning as above applies to the tiling once the column for $\zh$ is removed.  So in this case $\#\cO=l+1$.

Now let $k$ be the number of orbits.  From what we have just proved,  $\#\cA(S)=1+kl$.  Also, it follows easily from Lemma~\ref{num:A} that $\#\cA(S)=1+\al_1\cdots\al_n$.  Equating the two expressions results in the desired count.

\medskip

(b) We will consider the case $x\in B_i$ as the other is trivial.  Consider the tiling $\tau=\tau(\cO)$.  From the proof of (a), we see that row $i$ has $l$ columns which are tiled by a pair of consecutive black and yellow tiles of combined length $\al_i$.  So the number of black tiles in that row is
\begin{equation}
\label{l/al_i}
  m_i = l/\al_i.  
\end{equation}
We are now done by Corollary~\ref{I(T)} (a).

\medskip

(c)  Using part (b) and Corollary~\ref{I(T)} (b) we obtain
$$
\chi(\cO) = \be_{[n]} m_{[n]} + \sum_{i=1}^n \be_i m_i
=\delta+\sum_{i=1}^n \frac{l}{\alpha_i}(\alpha_i-1).
$$

\medskip

(d)  Again, this is easy to see if $x=\zh$.  If $x=x_{i,j}$ then there is no $J\subset [i,i]$ in $\cI(S)$.  So by Corollary~\ref{I(T)}  (c) and equation~\eqref{l/al_i}
$$
\chih_x(\cO) = j\cdot m_i = j l/\al_i.
$$

\medskip

(e)  It suffices to calculate the terms in the sum of Corollary~\ref{I(T)} (d).  We will do the case when $\zh\not\in\cO$ as the unique orbit when $\zh\in\cO$ is done similarly.  We first look at the term for $I=[n]$.  In this case $\be_{[n]}=1$ and $m_{[n]}=0$ by the choice of $\cO$.  Since $[i,i]\subset[n]$ for all $i$ and there is no column for the empty antichain we have $c_{[n]}=l$, the number of columns of the tiling.  So the term for $I=[n]$ reduces to $l$.  Now consider the summand for $[i,i]$.  We have $\be_i+1=\al_i$ and $m_i=l/\al_i$ by equation~\eqref{l/al_i}.  Furthermore, there is no $J\subset [i,i]$ so $c_i=0$.  Thus the term for $I=[i,i]$ is  the $i$th one in the sum given in (e), as desired.
\end{proof}

Stars exhibit a number of homomesies.  The following results are all gotten by simple manipulation of the formulas for $\chi$ and $\chih$ in the previous theorem, so we suppress the demonstration.
\begin{cor}
\label{start:hom}
Consider the star $S=S(\al_1,\ldots,\al_n)$.
\begin{enumerate}
    \item[(a)] If $x\in B_i$, then $\alpha_i\chi_x+\chi_{\hat{0}}$ is $1$-mesic.
    \item[(b)] If $x\in B_i$ and $y\in B_j$  then $\alpha_{i}\chi_{x}-\alpha_{j}\chi_{y}$ is $0$-mesic. 
    \item[(c)] If $x=x_{i,k}$ then $\al_i\chih_x - k\chih_{\zh}$ is $0$-mesic.
    \item[(d)]  If $x=x_{i,k}$ and $y=x_{j,k}$, then $\al_i\chih_x-\al_j\chih_y$ is $0$-mesic.\hqed
\end{enumerate}
\end{cor}

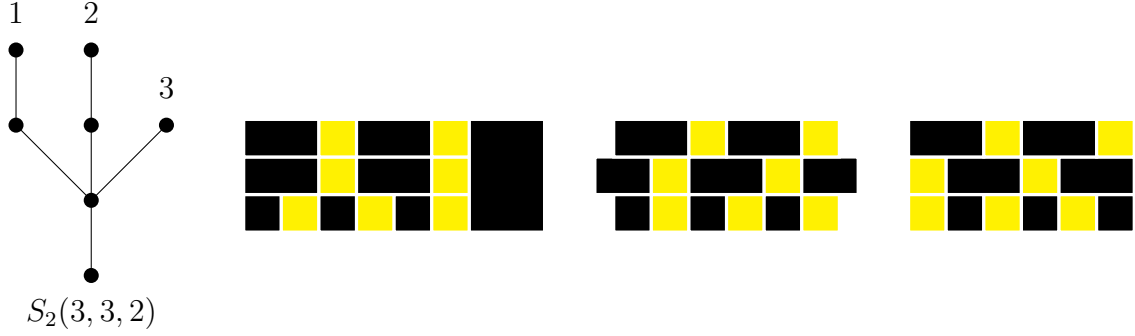
\begin{figure}
 \begin{tikzpicture}
 \fill(1,-1) circle(.1);
\fill(1,0) circle(.1);
\fill(0,1) circle(.1);
\fill(1,1) circle(.1);
\fill(2,1) circle(.1);
\fill(0,2) circle(.1);
\fill(1,2) circle(.1);
\draw (0,2)--(0,1)--(1,0)--(1,2) (1,-1)--(1,0)--(2,1);
\draw(0,2.5) node{$1$};
\draw(1,2.5) node{$2$};
\draw(2,1.5) node{$3$};
\draw(1,-1.5) node{$S_2(3,3,2)$};
 \end{tikzpicture}
\hs{10pt}
\raisebox{40pt}{
 \begin{tikzpicture}[scale=.5]
 \blr{(0,0)}{(1,1)}
 \yer{(1,0)}{(2,1)}
  \blr{(2,0)}{(3,1)}
 \yer{(3,0)}{(4,1)}
  \blr{(4,0)}{(5,1)}
 \yer{(5,0)}{(6,1)}
 \blr{(0,1)}{(2,2)}
 \yer{(2,1)}{(3,2)}
  \blr{(3,1)}{(5,2)}
  \yer{(5,1)}{(6,2)}
 \blr{(0,2)}{(2,3)}
 \yer{(2,2)}{(3,3)}
  \blr{(3,2)}{(5,3)}
  \yer{(5,2)}{(6,3)}
 \blr{(6,0)}{(8,3)}
 \end{tikzpicture}
 \hs{10pt}
 \begin{tikzpicture}[scale=.5]
 \blr{(0,0)}{(1,1)}
 \yer{(1,0)}{(2,1)}
  \blr{(2,0)}{(3,1)}
 \yer{(3,0)}{(4,1)}
  \blr{(4,0)}{(5,1)}
 \yer{(5,0)}{(6,1)}
 \blr{(-.5,1)}{(1,2)}
 \yer{(1,1)}{(2,2)}
  \blr{(2,1)}{(4,2)}
  \yer{(4,1)}{(5,2)}
 \blr{(5,1)}{(6.5,2)}  
 \blr{(0,2)}{(2,3)}
 \yer{(2,2)}{(3,3)}
  \blr{(3,2)}{(5,3)}
  \yer{(5,2)}{(6,3)}
 \end{tikzpicture}
 \hs{10pt}
 \begin{tikzpicture}[scale=.5]
 \yer{(0,0)}{(1,1)}
 \blr{(1,0)}{(2,1)}
  \yer{(2,0)}{(3,1)}
 \blr{(3,0)}{(4,1)}
  \yer{(4,0)}{(5,1)}
 \blr{(5,0)}{(6,1)}
\yer{(0,1)}{(1,2)}
 \blr{(1,1)}{(3,2)}
 \yer{(3,1)}{(4,2)}
  \blr{(4,1)}{(6,2)}
 \blr{(0,2)}{(2,3)}
 \yer{(2,2)}{(3,3)}
  \blr{(3,2)}{(5,3)}
  \yer{(5,2)}{(6,3)}
 \end{tikzpicture}
 }
    \caption{The extended star $S_2(3,3,2)$  and its tilings}
    \label{S_2(3,3,2)}
\end{figure}

It is easy to generalize Theorem~\ref{star:thm}  to the case where $b_{[n]}>1$ so that one has a fatter $[n]$-tile. More generally, we will describe what happens to any tree where $\zh$ is covered by a single element.  An example can be obtained by comparing Figures~\ref{S(3,3,2)} and~\ref{S_2(3,3,2)}.
\begin{prop}
\label{add:zh}
Suppose $T\setm\{\zh\}=T'$ is a rooted tree with $n$ leaves.  Let the $\cI(T)$-tilings be $\tau_1,\tau_2,\ldots,\tau_k$ where $\tau_1$ is the tiling for the orbit of $\zh$.
Then the $\cI(T')$-tilings are $\tau_1',\tau_2,\ldots,\tau_k$ where $\tau_1'$ is obtained from $\tau_1$ by widening the $[n]$-tile by one column.
\end{prop}
\begin{proof}
Since $T\setm\{\zh\}=T'$, the intervals of $T$ and $T'$ are the same. Also 
$$
\be_I(T')=\case{\be_I(T)}{if $I\neq[n]$,}{\be_{[n]}(T)+1}{if $I=[n]$.} 
$$
Definition~\ref{tau:def} now shows that the tilings transform as desired.
\end{proof}

For a positive integer $b$ the {\em $b$-extended star}, $S_b(\al_1,\ldots,\al_n)$, is the rooted tree with
$$
\cI(S_b) = \{([1,1],\be_1),\ \ldots,\ ([n,n],\be_n),\ ([n],b)\}
$$
and $\al_i=\be_i+1$ for $i\in[n]$.  So we recover ordinary stars when $b=1$.
We see $S_2(3,3,2)$ in Figure~\ref{S_2(3,3,2)}.  The next result follows easily from Theorem~\ref{star:thm} and Proposition~\ref{add:zh} and so the proof is omitted.
\begin{cor}
\label{estar}
Consider the extended star $S_b=S_b(\al_1,\ldots,\al_n)$ and an orbit $\cO$ of rowmotion on $S_b$.  Let $l=\lcm(\al_1,\ldots,\al_n)$.  
\begin{enumerate}
\item[(a)] We have 
\[
\#\cO=l+\de b
\]
and the number of orbits is $\al_1\cdots\al_n/l$.
    \item[(b)] For any $x\in S$,
    $$\chi_x(\cO)=\begin{cases} 
      l/\alpha_i & \text{if $x\in B_i$},\\
     \delta & \text{if $x\in B_{[n]}$}.
   \end{cases}$$
    \item[(c)] We have
    $$
       \chi(\cO)=\delta b+\sum_{i=1}^n \frac{l}{\alpha_i}(\alpha_i-1).
    $$
    Thus $\chi$ is homometric but not homomesic.
    \item[(d)] For any $x\in S$
   $$
   \chih_x(\cO)=\case{jl/\alpha_i}{if $x=x_{i,j}$}{l+\de(j-1)}{if $x=x_{[n],j}$.}
     $$
     \item[(e)]  We have
     $$
        \hat{\chi}(\cO)= l b + \de \binom{b}{2}
        +\sum_{i=1}^n \frac{l}{\alpha_i}\binom{\alpha_i}{2}.
     $$
     Thus $\chih$ is homometric but not homomesic.\hqed
\end{enumerate}
\end{cor}


\section{Trees with three leaves}
\label{ttl}

\begin{figure}
 \begin{tikzpicture}
\fill(1,0) circle(.1);
\fill(0,1) circle(.1);
\fill(2,1) circle(.1);
\draw (0,1)--(1,0)--(2,1);
\draw(1,-.5) node{$T'=T''$};
 \end{tikzpicture}
\hs{20pt}
 \begin{tikzpicture}[scale=.5]
\blr{(0,0)}{(1,2)}
\blr{(1,0)}{(2,1)}
\blr{(1,1)}{(2,2)}
\yer{(2,0)}{(3,1)}
\yer{(2,1)}{(3,2)}
\draw(1.5,-1) node{$\tau_1'=\tau_1''$};
 \end{tikzpicture}
 \hs{20pt}
 \begin{tikzpicture}[scale=.5]
 \yer{(0,0)}{(1,1)}
 \blr{(0,1)}{(1,2)}
  \blr{(1,0)}{(2,1)}
 \yer{(1,1)}{(2,2)}
 \draw(1,-1) node{$\tau_2'=\tau_2''$};
 \end{tikzpicture}
 
 \vs{20pt}
 
  \begin{tikzpicture}
 \fill(1,-1) circle(.1);
\fill(1,0) circle(.1);
\fill(0,1) circle(.1);
\fill(2,1) circle(.1);
\fill(-1,2) circle(.1);
\fill(.5,2) circle(.1);
\fill(1.5,2) circle(.1);
\fill(3,2) circle(.1);
\draw (1,-1)--(1,0) (-1,2)--(1,0)--(3,2) (0,1)--(.5,2) (2,1)--(1.5,2);
\draw(1,-1.5) node{$T$};
 \end{tikzpicture}
\hs{10pt}
\raisebox{40pt}{
$\barr{c}
 \begin{tikzpicture}[scale=.5]
\blr{(0,0)}{(1,2)}
\blr{(0,2)}{(1,4)}
\blr{(1,0)}{(2,1)}
\blr{(1,1)}{(2,2)}
\blr{(1,2)}{(2,3)}
\blr{(1,3)}{(2,4)}
\yer{(2,0)}{(3,1)}
\yer{(2,1)}{(3,2)}
\yer{(2,2)}{(3,3)}
\yer{(2,3)}{(3,4)}
\blr{(3,0)}{(5,4)}
\draw(2.5,-1) node{$\tau^{1,1}_1$};
 \end{tikzpicture}
 \hs{10pt}
 \begin{tikzpicture}[scale=.5]
\blr{(1,0)}{(2,2)}
\blr{(0,2)}{(1,4)}
\blr{(2,0)}{(3,1)}
\blr{(2,1)}{(3,2)}
\blr{(1,2)}{(2,3)}
\blr{(1,3)}{(2,4)}
\yer{(0,0)}{(1,1)}
\yer{(0,1)}{(1,2)}
\yer{(2,2)}{(3,3)}
\yer{(2,3)}{(3,4)}
\draw(2.5,-1) node{$\tau^{1,1}_2$};
 \end{tikzpicture}
 \hs{10pt}
  \begin{tikzpicture}[scale=.5]
\blr{(2,0)}{(3,2)}
\blr{(0,2)}{(1,4)}
\blr{(0,0)}{(1,1)}
\blr{(0,1)}{(1,2)}
\blr{(1,2)}{(2,3)}
\blr{(1,3)}{(2,4)}
\yer{(1,0)}{(2,1)}
\yer{(1,1)}{(2,2)}
\yer{(2,2)}{(3,3)}
\yer{(2,3)}{(3,4)}
\draw(2.5,-1) node{$\tau^{1,1}_3$};
 \end{tikzpicture}
 \hs{30pt}
 \begin{tikzpicture}[scale=.5]
 \yer{(0,0)}{(1,1)}
 \blr{(0,1)}{(1,2)}
  \yer{(0,2)}{(1,3)}
 \blr{(0,3)}{(1,4)}
  \blr{(1,0)}{(2,1)}
 \yer{(1,1)}{(2,2)}
 \blr{(1,2)}{(2,3)}
 \yer{(1,3)}{(2,4)}
 \draw(1,-1) node{$\tau^{2,2}_1$};
 \end{tikzpicture}
 \hs{10pt}
  \begin{tikzpicture}[scale=.5]
 \blr{(0,0)}{(1,1)}
 \yer{(0,1)}{(1,2)}
  \yer{(0,2)}{(1,3)}
 \blr{(0,3)}{(1,4)}
  \yer{(1,0)}{(2,1)}
 \blr{(1,1)}{(2,2)}
 \blr{(1,2)}{(2,3)}
 \yer{(1,3)}{(2,4)}
 \draw(1,-1) node{$\tau^{2,2}_2$};
 \end{tikzpicture}
 \\[5pt]
 \begin{tikzpicture}[scale=.5]
\blr{(0,2)}{(1,4)}
\blr{(1,2)}{(2,3)}
\blr{(1,3)}{(2,4)}
\yer{(2,2)}{(3,3)}
\yer{(2,3)}{(3,4)}
\blr{(3,2)}{(4,4)}
\blr{(4,2)}{(5,3)}
\blr{(4,3)}{(5,4)}
\yer{(5,2)}{(6,3)}
\yer{(5,3)}{(6,4)}
 \yer{(0,0)}{(1,1)}
 \blr{(0,1)}{(1,2)}
  \blr{(1,0)}{(2,1)}
 \yer{(1,1)}{(2,2)}
  \yer{(2,0)}{(3,1)}
 \blr{(2,1)}{(3,2)}
  \blr{(3,0)}{(4,1)}
 \yer{(3,1)}{(4,2)}
   \yer{(4,0)}{(5,1)}
 \blr{(4,1)}{(5,2)}
  \blr{(5,0)}{(6,1)}
 \yer{(5,1)}{(6,2)}
\draw(3,-1) node{$\tau^{1,2}_1$};
 \end{tikzpicture}
 \hs{30pt}
  \begin{tikzpicture}[scale=.5]
\blr{(0,0)}{(1,2)}
\blr{(1,0)}{(2,1)}
\blr{(1,1)}{(2,2)}
\yer{(2,0)}{(3,1)}
\yer{(2,1)}{(3,2)}
\blr{(3,0)}{(4,2)}
\blr{(4,0)}{(5,1)}
\blr{(4,1)}{(5,2)}
\yer{(5,0)}{(6,1)}
\yer{(5,1)}{(6,2)}
\yer{(0,2)}{(1,3)}
\blr{(0,3)}{(1,4)}
\blr{(1,2)}{(2,3)}
\yer{(1,3)}{(2,4)}
\yer{(2,2)}{(3,3)}
\blr{(2,3)}{(3,4)}
\blr{(3,2)}{(4,3)}
\yer{(3,3)}{(4,4)}
\yer{(4,2)}{(5,3)}
\blr{(4,3)}{(5,4)}
\blr{(5,2)}{(6,3)}
\yer{(5,3)}{(6,4)}
\draw(3,-1) node{$\tau^{2,1}_1$};
 \end{tikzpicture}
\earr$ 
}
    \caption{The trees $T',T'',T$ and their tilings}
    \label{two:fig}
\end{figure}
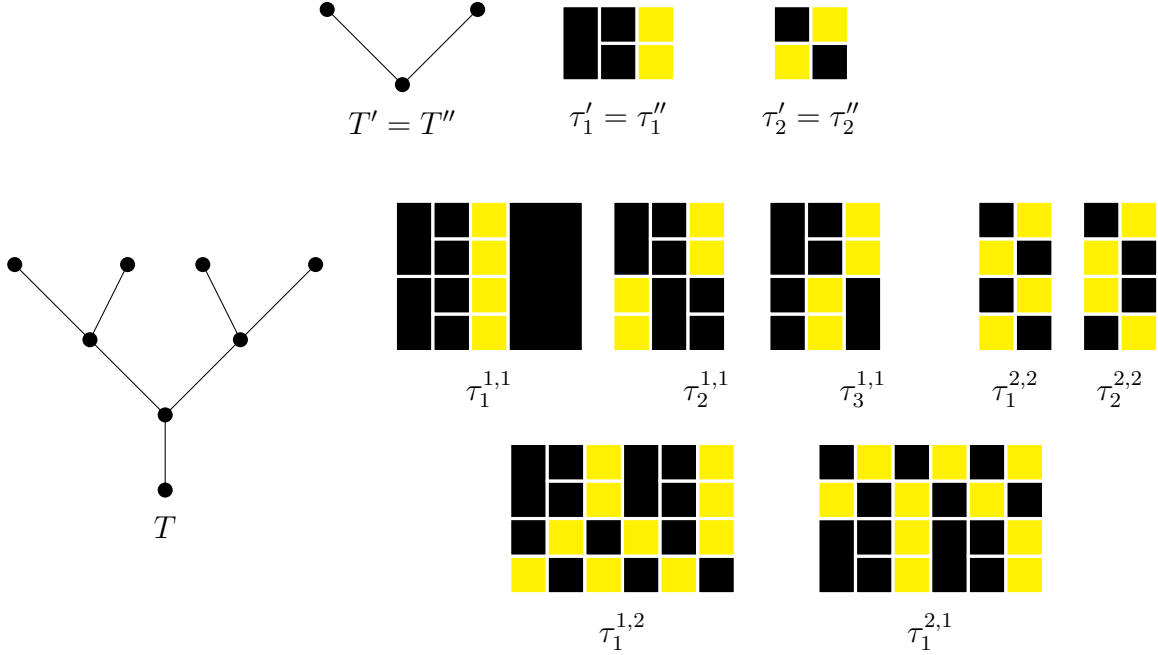

The special case $n=3$ of \Cref{estar} gives information about the rowmotion orbits on trees that have three leaves whose branches have minimal elements covering a single vertex of the tree.  Up to isomorphism, there is only one other arrangement of branches in a tree with three leaves and this section is devoted to studying this case.  First, we will prove a result about removing the branch containing $\zh$ from a certain type of tree.

\Cref{add:zh} describes the tilings of a tree $T$ whose $\zh$ is covered by a single element.  We will determine what happens when it is covered by two elements or, more generally, when removing the branch of $\zh$ leaves exactly two rooted trees remaining.  It is possible to derive a similar result for any number of rooted subtrees, but the notation becomes cumbersome and we will only need the case of two subtrees in the sequel.

In order to state our result we will need some notation. 
Let $T$ be a rooted tree such that $T\setm B=T'\uplus T''$ where $B$ is the branch of $\zh$ and $T',T''$ are rooted trees.  Suppose that $T'$ has $n'$ leaves and tilings $\tau_1',\ldots,\tau_s'$ where $\tau_1'$ is corresponds to the orbit containing $\zh'$, the minimal element of $T'$.  Further, let $c_i'$ be the number of columns of $\tau_i'$ for $i\in[s]$. Notation used previously for $T$ will be given a single prime when applied to $T'$.  Similarly, let $T''$ have $n''$ leaves and tilings $\tau_1'',\ldots,\tau_t''$ with the same conventions about the tilings and other notation except with a double prime.  An example of this construction can be found in \Cref{two:fig}.
\begin{thm}
\label{two:thm}
Let $T$ be a rooted tree with $T\setm B=T'\uplus T''$ as above.
\begin{enumerate}
    \item[(a)]  The tilings of $T$ can be described as follows.  For all $(i,j)\in[n']\times[n'']$ there are  tilings $\tau^{i,j}_m$ for $1\le m\le g_{i,j}:=\gcd(c_i',c_j'')$.  Unless $i=j=m=1$, we have that $\tau^{i,j}_m$ consists of consecutive copies of $\tau_i'$ in the first $n'$ rows, consecutive copies of $\tau_j''$ in the last $n''$ rows, and has $l_{i,j}:=lcm(c_i',c_j'')$ columns.
    Tiling $\tau^{1,1}_1$ is as in the previous sentence except that one copy of $\tau_1'$ and one of $\tau_1''$ align so that their columns of all yellow tiles coincide, and an $[n'+n'']\times b$ black tile is inserted directly after that column to make the total length of the orbit $l_{1,1}+b$ where $b=\#B$.
    \item[(b)]  Let $\cO_i'$, $\cO_j''$, and $\cO^{i,j}_m$ be the orbits corresponding to tilings $\tau_i'$, $\tau_j''$, and $\tau^{i,j}_m$, respectively.  For any $x\in T$
    $$
    \chi_x(\cO^{i,j}_m)=
    \begin{cases}
l_{i,j}\chi_x(\cO_i')/c_i' & \text{if $x\in T'$,}\\
l_{i,j}\chi_x(\cO_j'')/c_j'' & \text{if $x\in T''$,}\\
\de &\text{if $x\in B$.}
    \end{cases}
    $$
\item[(c)]  We have
$$
 \chi(\cO^{i,j}_m)=\de b + l_{i,j}\chi(\cO_i')/c_i' + l_{i,j}\chi(\cO_j'')/c_j''.
$$
\item[(d)]  For any $x\in T$
    $$
    \chih_x(\cO^{i,j}_m)=
    \begin{cases}
l_{i,j}\chih_x(\cO_i')/c_i' & \text{if $x\in T'$,}\\
l_{i,j}\chih_x(\cO_j'')/c_j'' & \text{if $x\in T''$,}\\
l_{i,j}+\de(j-1) &\text{if $x=x_{[n'+n''],j}$.}
    \end{cases}
    $$
\item[(e)]  We have
$$
 \chih(\cO^{i,j}_m)=l_{i,j}b + \de\binom{b}{2} + l_{i,j}\chih(\cO_i')/c_i' + l_{i,j}\chih(\cO_j'')/c_j''.
$$
\end{enumerate}
\end{thm}
\begin{proof}
We will only prove (a), as once this is established then the other parts of the theorem follow from straight-forward computations similar to those already seen in \Cref{star:thm}.  Let $\cO$ be an antichain orbit of $T$. Pick an antichain $A$ in $\cO$ which does not contain an element of $B$, so that it  can be written as $A=A'\uplus A''$ where $A'=A\cap T'$ and $A''=A\cap T''$.  Let $\cO'$ and $\cO''$ be the orbits of $A'$ and $A''$ in $T'$ and $T''$, respectively. 

First consider the case when (at least) one of $\cO'$ and $\cO''$ does not contain the empty antichain.   It follows that as $\rho$ is applied to $A$, the antichains $A'$ and $A''$ will describe their respective orbits $\cO'$ and $\cO''$ in $T'$ and $T''$.  
If $c'=\#\cO'$ and $c''=\#\cO''$ then, in order for both orbits to return to $A'$ and $A''$ at the same time, we must have $\#\cO=\lcm(c',c'')$.  And since there are $c' c''$ ways to pair an antichain in $\cO'$ with one in $\cO''$, the total number of orbits obtained from such pairs is $c' c''/\lcm(c',c'') = \gcd(c',c'')$.
This description matches the one given for the tilings $\tau^{i,j}_k$ for as long as we do not have $i=j=1$.

In the case when both $\cO'$ and $\cO''$ contain the empty antichain, the argument of the previous paragraph goes through with one exception.  Suppose the elements of $\cO'$ and $\cO''$ are repeated in $\cO$ in such a way that at some point the empty antichain of $T$ is reached.  Then  $\emp$ will be followed by the elements of $B$ in increasing order.  This, in turn, will be followed by the antichain $\{\zh',\zh''\}$ which will cause the orbits $\cO'$ and $\cO''$ to continue.  This orbit corresponds to the tiling $\tau^{1,1}_1$ and completes our description of the orbits and their tilings.
\end{proof}

\begin{figure}
 \begin{tikzpicture}
\fill(4,0) circle(.1);
\fill(4,1) circle(.1);
\fill(4,2) circle(.1);
\fill(3,3) circle(.1);
\fill(5,3) circle(.1);
\fill(2,4) circle(.1);
\fill(5,4) circle(.1);
\fill(1,5) circle(.1);
\fill(0,6) circle(.1);
\fill(2,6) circle(.1);
\fill(0,7) circle(.1);
\fill(2,7) circle(.1);
\draw (4,0)--(4,2)--(0,6)--(0,7) (1,5)--(2,6)--(2,7) (4,2)--(5,3)--(5,4);
\draw(0,7.5) node{$1$};
\draw(2,7.5) node{$2$};
\draw(5,4.5) node{$3$};
\draw(4,-.5) node{$T_3$};
 \end{tikzpicture}
    \caption{The tree $T_3$}
    \label{T_3}
\end{figure}
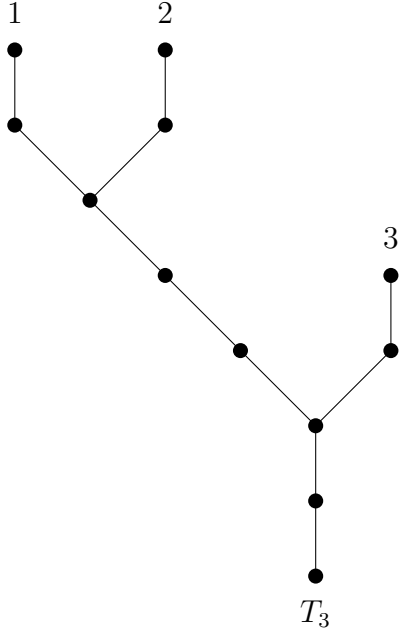

Now consider a tree $T$ with three leaves which is not an extended star.  It follows that, using a suitable embedding, we will have
$$
\cI(T) = \{ ([3], a),\ ([2],b),\ ([1,1],c),\ ([2,2],d),\ ([3,3],e)\}
$$
for $a,b,c,d,e\ge1$.  A particular tree of this form is shown in \Cref{T_3}.  Although we can use the previous theorem to calculate the orbits and their statistic values for arbitrary $a,b,c,d,e$ the resulting formulas are not very enlightening.  So we will concentrate on a specific tree of this type.  Define the three-leaf tree $T_k$ to be the one with
$$
\cI(T) = \{ ([3], k),\ ([2],k),\ ([1,1],k-1),\ ([2,2],k-1),\ ([3,3],k-1)\}.
$$
The tree in \Cref{T_3} is $T_3$.
\begin{thm}
The orbits of rowmotion on $T_k$ can be partitioned by length into three sets $\cS$ (for small), $\cM$ (for medium), and $\cL$ (for large) with the following properties.
\begin{enumerate}
    \item[(a)] We have 
    $$
    \#\cS=k(k-1),\quad\#\cM=k-1,\quad\#\cL=1,
    $$
    and 
    $$
    \#\cO =
    \begin{cases}
    k&\text{if $\cO\in\cS$},\\
    2k&\text{if $\cO\in\cM$},\\
    3k&\text{if $\cO\in\cL$}.
    \end{cases}
    $$
    \item[(b)] We have
    $$
    \chi(\cO) =
    \begin{cases}
    3k-3 &\text{if $\cO\in\cS$},\\
    5k-4 &\text{if $\cO\in\cM$},\\
    6k-4 &\text{if $\cO\in\cL$}.
    \end{cases}
    $$
    Thus $\chi$ is homometric but not homomesic.
    \item[(c)] We have
    $$
    \hat{\chi}(\cO) =
    \begin{cases}
    \frac{7}{2}k^2-\frac{3}{2}k &\text{if $\cO\in\cS$},\\[1ex]
    \frac{11}{2}k^2-\frac{5}{2}k &\text{if $\cO\in\cM$},\\[1ex]
    6k^2-3k &\text{if $\cO\in\cL$}.
    \end{cases}
    $$
    Thus $\hat{\chi}$ is homometric but not homomesic.
\end{enumerate}
\end{thm}
\begin{proof}
(a)  Let $B=B_{[3]}$ and $b=\#B=k$.  Then $T_k\setm B= S_k(k,k) \uplus S(k)$ is a disjoint union of two (extended) stars.  Clearly $T''=S(k)$ has only one orbit which contains $\zh$.
By \Cref{estar}, $T'=S_k(k,k)$ has orbits of size $\lcm(k,k)=k$ and the total number of orbits is $(k\cdot k)/k = k$.  So one of these orbits contains $\zh$ and the other $k-1$ do not, and they will have lengths given by $k+\de k$. 
It follows that the latter will be of length $k$ while the former is of length $2k$.
Applying \Cref{two:thm}, $T_k$ will have $k(k-1)$ orbits $\cO^{i,1}_m$ with $i\neq 1$ and these will have length $\lcm(k,k)=k$.  These are the orbits in $\cS$.  There will also be the orbits $\cO^{1,1}_m$ for $m\in[2,k]$  which gives $k-1$ possible values for $m$.  Here the length is $\lcm(2k,k)=2k$.  These are the orbits in $\cM$.  Finally, the unique orbit $\cO^{1,1}_1$ is of length $2k+b=2k+k=3k$ and this describes $\cL$.

\medskip

(b)  We will do the case of $\cO^{1,1}_1$, the unique element of $\cL$, as the others are similar.  Applying \Cref{estar} (c) to orbit $\cO_1'$ of $T'=S_k(k,k)$ gives
$$
\chi(\cO_1') = k +\sum_{i=1}^2 \frac{k}{k}(k-1) = 3k-2
$$
Similarly, for $\cO_1''$ in $T''=S(k)$ we have
$$
\chi(\cO_1'') = k-1.
$$
Now applying \Cref{two:thm} (c) with $l_{1,1} =\lcm(2k,k)=2k$ yields
$$
\chi(\cO^{1,1}_1) = k + 2k(3k-2)/(2k) + 2k(k-1)/k = 6k-4.
$$

\medskip
(c)  The computations are like those in (b) except using \Cref{two:thm} (e), so the details are omitted.
\end{proof}

\begin{figure}
 \begin{tikzpicture}
\fill(3,0) circle(.1);
\fill(2,1) circle(.1);
\fill(4,1) circle(.1);
\fill(1,2) circle(.1);
\fill(3,2) circle(.1);
\fill(0,3) circle(.1);
\fill(2,3) circle(.1);
\draw (0,3)--(3,0)--(4,1) (2,1)--(3,2) (1,2)--(2,3);
\draw(0,3.5) node{$1$};
\draw(2,3.5) node{$2$};
\draw(3,2.5) node{$3$};
\draw(4,1.5) node{$4$};
\draw(3,-.5) node{$C_3$};
 \end{tikzpicture}
 \hs{20pt}
  \begin{tikzpicture}
\fill(3,0) circle(.1);
\fill(2,1) circle(.1);
\fill(4,1) circle(.1);
\fill(1,2) circle(.1);
\fill(0,3) circle(.1);
\fill(2,3) circle(.1);
\fill(-1,4) circle(.1);
\fill(-2,5) circle(.1);
\fill(0,5) circle(.1);
\draw (-2,5)--(2,1) (2,1)--(3,0)--(4,1)  (1,2)--(2,3) (-1,4)--(0,5);
\draw(-2,5.5) node{$1$};
\draw(0,5.5) node{$2$};
\draw(2,3.5) node{$3$};
\draw(4,1.5) node{$4$};
\draw(3,-.5) node{$C_{3,2}$};
 \end{tikzpicture}
    \caption{The comb $C_3$ and extended comb $C_{3,2}$}
    \label{comb:fig}
\end{figure}
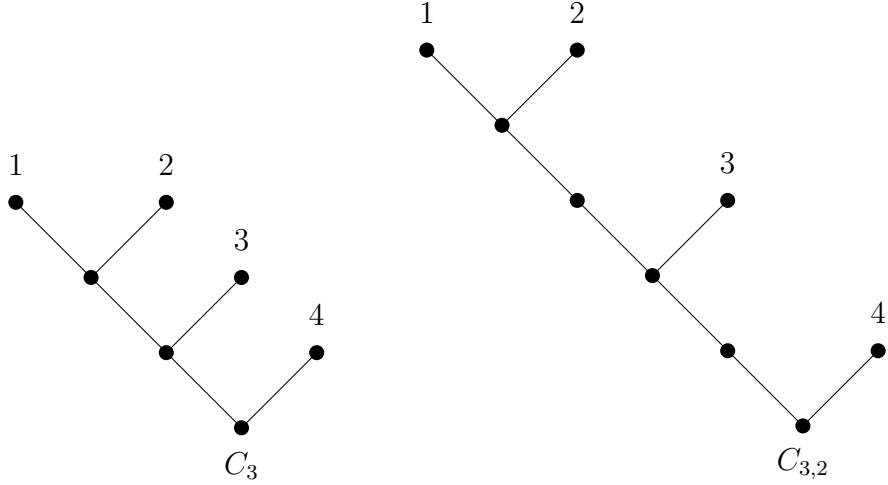

\section{Combs and zippers}
\label{cz}

Combs are a particularly simple type of binary tree.  They are useful in understanding the structure of the free Lie algebra as shown, for example, in the work of Wachs~\cite{wac:cpl}.  In this section we will compute the orbit structure of combs, combs with an extended backbone, and zippers which are constructed by pasting together combs.

It will be convenient to consider combs which have $n+1$ leaves.  Specifically, the {\em comb}, $C_n$, is the rooted tree with
$$
\cI(C_n)=\{ ([n+1],1),\ ([n],1),\ \ldots,\ ([2],1),
([1,1],1),\ ([2,2],1),\ \ldots,\ ([n+1,n+1],1)\}.
$$
The comb $C_3$ is shown on the left in Figure~\ref{comb:fig}. 
\begin{thm}
\label{C_n:thm}
The orbits of rowmotion on $C_n$ can be partitioned into two sets $\mathcal{S}$ and $\mathcal{L}$  having the following properties.
\begin{enumerate}
\item[(a)]
We have
    $$
    \# \cS=2^{n-1},\quad\#\cL=1,
    $$
    and
    $$
    \#\cO =
    \begin{cases}
   2 &\text{if $\cO\in\cS$},\\
   2^{n+1}-1 &\text{if $\cO\in\cL$}.
    \end{cases}
    $$
\item[(b)]
We have
    $$
    \chi(\cO) =
    \begin{cases}
    n+1 &\text{if $\cO\in\cS$},\\
    (2n+1)2^{n-1} &\text{if $\cO\in\cL$}.
    \end{cases}
    $$
Thus $\chi$ is homometric but not homomesic.
\item[(c)]
We have
    $$
    \hat{\chi}(\cO) =
    \begin{cases}
    3n+1 &\text{if $\cO\in\cS$},\\
    2^{n-1}(6n-5) + 3 &\text{if $\cO\in\cL$}.
    \end{cases}
    $$
Thus $\hat{\chi}$ is homometric but not homomesic.
\end{enumerate}
\end{thm}
\begin{proof}
(a)  We induct on $n$ where the result is easy to check if $n=1$.  Assume the orbits are as stated for $C_n$ and that the unique orbit in $\cL$ is the one containing $\zh$.
We see that $C_{n+1}\setm\{\zh\} = C_n\uplus\{v\}$ where $v$ is the leaf labeled $n+2$.
We will subscript notation with $n$ or $n+1$ to make it clear which comb is meant.

Now $T''=\{v\}$ has only one orbit of length $2$.  By \Cref{two:thm}, this combines with each of the orbits in $\cS_n$ to give orbits of length $\lcm(2,2)=2$.  Also, there will be $\gcd(2,2)=2$ orbits in $\cS_{n+1}$ for every one in $\cS_n$ for a total of $2\cdot 2^{n-1}=2^n$ orbits.  Thus the information about $\cS_{n+1}$ is as desired.

The one orbit in $\cL_n$ will combine with the one for $\{v\}$ to give $\gcd(2,2^{n+1}-1)=1$ orbit which must be the one containing $\zh$.  So its length will be
$\lcm(2,2^{n+1}-1)+1 = 2^{n+2}-1$, which finishes the induction.

\medskip

(b)  Again we induct,  only providing details for the orbit of $\zh$ in $\cL$.  Using the notation for \Cref{two:thm} we have $c_1'=2^{n+1}-1$ and $c_1''=2$.  So $l_{1,1}= c_1' c_1''$ and the formula in part (c) of that theorem becomes
$$
\chi(\cO)=1+2(2n+1)2^{n-1} + (2^{n+1}-1)\cdot 1 = 
(2n+3) 2^n
$$
as it should be.

\medskip

(c)  This demonstration is similar to that of (b) above using \Cref{two:thm} (d) and so is omitted.
\end{proof}

We can generalize these comb results as follows.  The {\em backbone} of a comb is the set of elements which are not leaves.  So $C_n$ has an $n$-element backbone and each element is an interval in $\cI(C_n)$.  We will extend each of these intervals, except for the one corresponding to $\zh$, so that they have $k$ elements.  Formally, the {\em extended comb}, $C_{n,k}$, is defined as the tree with
$$
\cI(C_n)=\{ ([n+1],1),\ ([n],k),\ \ldots,\ ([2],k),
([1,1],1),\ ([2,2],1),\ \ldots,\ ([n+1,n+1],1)\}.
$$
On the right in Figure~\ref{comb:fig} is the extended comb $C_{3,2}$.
Note that $C_{n,1}=C_n$.
\begin{thm}
The orbits of the extended comb $C_{n,k}$ can be partitioned into two sets $\cS$ and $\cL$ when $k$ is odd, and into $n+1$ sets $\cS_1,\cS_2,\ldots,\cS_n$ and $\cL$ when $k$ is even.  The orbits have properties given by the following tables for $k$ odd:
\bce
    \begin{tabular}{|c|c|c|}
    \hline
         $k$ odd & $\cS$ &$\cL$ \\ \hline
        $\#\cO$ & $2$ & $(k+1)2^n-2k+1$ \\
        number of $\cO$ & $2^{n-1}$ & $1$\\ 
        $\chi(\cO)$ & $n+1$ &$((k+1)n+1)2^{n-1}-k+1$ \\
        $\hat{\chi}(\cO)$ & $(2k+1)n-2k+3$ &$(2k+1)(k+1)n 2^{n-1}-(5k^2+3k-3)2^{n-1}+3k^2$ \\\hline
    \end{tabular}
\ece
and for $k$ even:
\bce
    \begin{tabular}{|c|c|c|}
    \hline
         $k$ even & $\cS_i$ for $i\in[n]$ &$\cL$ \\ \hline
        $\#\cO$ & $k(i-1)+2$ & $k(n-1)+3$ \\
        number of $\cO$ & $2^{n-i}$ & $1$\\ 
        $\chi(\cO)$ & $\frac{k(i-1)+2}{2}n-\frac{k}{4}(i^2-5i+4)+1$ & $\frac{k}{4}n^2+\frac{3k+4}{4}n-k+2$ \\[5pt]
        $\hat{\chi}(\cO)$
&
$\frac{(2k+1)(k(i-1)+2)}{2}n-\frac{k(2k+1)}{4}i^2+\frac{3k}{4}i+\binom{k-2}{2}$ & 
$\frac{k(2k+1)}{4}n^2-\frac{4k^2-9k-4}{4}n+\binom{k-2}{2}$ \\[5pt]
        \hline
    \end{tabular}
\ece
Thus $\chi$ and $\chih$ are homometric on $C_{n,k}$.
\end{thm}
\begin{proof}
We will just verify the orbit structure as, once that is done, the calculation of $\chi$ and $\chih$ are routine using \Cref{add:zh}
and \Cref{two:thm}.  We will induct on $n$ where the base case is easy. 
Note that $C_{n+1,k}\setm\{\zh\}=C_{n,k}'\uplus\{v\}$ where $v$ is the leaf labeled $n+2$ and $C_{n,k}'$ is $C_{n,k}$ with its $\zh$-interval replaced by one with $k$ elements.  It follows from \Cref{add:zh} that the orbits of these two posets are identical except for the orbit of $\zh$ whose $[n+1]$-tile has been widened by adding $k-1$ columns.

We now consider what happens when $k$ is odd.  The orbits of length $2$ for $C_{n,k}'$ combine with the orbit of length $2$ for $\{v\}$ in exactly the same way as in the proof of \Cref{C_n:thm}.  As far as the orbit containing $\zh$ in $C_{n,k}'$, by induction and the last sentence of the previous paragraph it has length
$$
[(k+1)2^n-2k+1]+k-1 = (k+1)2^n-k
$$
which is odd by the parity of $k$.  So, by \Cref{two:thm}, the orbit containing $\zh$ in $C_{n+1,k}$ has length
$$
\lcm((k+1)2^n-k,2) + 1 =
2[(k+1)2^n-k]+1 = (k+1)2^{n+1}-2k+1
$$
which is the desired quantity.

When $k$ is even we have, by induction, that all the orbits of $C_{n,k}$ have even length except for the orbit of $\zh$ whose length is odd.  It follows that all the orbits of $C_{n,k}'$ are of even length.  So, when each non-$\zh$ is  combined with $v$'s orbit of length $2$, this will result in two orbits of the same length.  This accounts for the orbits in $\cS_i$ of $C_{n+1,k}$ for $i<n$.  The $\zh$-orbit of $C_{n,k}'$ will have length
$$
[k(n-1)+3]+k-1 = kn+2.
$$
Since this is even, when it combines with $v$'s orbit it will produce $\gcd(kn+2,2)=2$ orbits for $C_{n+1,k}$.  One of these will be of size $\lcm(kn+2,2)=kn+2$ and that one will take care of $\cS_n$.  The other will have length one more and will be the orbit in $\cL$.
 \end{proof}
 
\begin{figure}
 \begin{tikzpicture}
\fill(2,0) circle(.1);
\fill(0,1) circle(.1);
\fill(4,1) circle(.1);  
\fill(0,2) circle(.1);
\fill(1,2) circle(.1); 
\fill(3,2) circle(.1);
\fill(4,2) circle(.1);
\fill(0,3) circle(.1);
\fill(1,3) circle(.1); 
\fill(3,3) circle(.1);
\fill(4,3) circle(.1); 
\fill(0,4) circle(.1);
\fill(1,4) circle(.1); 
\fill(3,4) circle(.1);
\fill(4,4) circle(.1); 
\draw (0,4)--(0,1)--(2,0)--(4,1)--(4,4) 
(0,1)--(1,2) (0,2)--(1,3) (0,3)--(1,4)
(4,1)--(3,2) (4,2)--(3,3) (4,3)--(3,4);
 \end{tikzpicture}
    \caption{The zipper $Z_3$}
    \label{Z_3}
\end{figure}
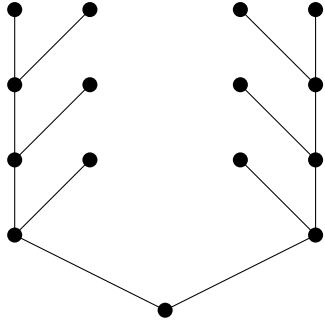
 
 Another way to modify combs is by combining them together.  If $T$ is a rooted tree and $T\setm\{\zh\}=T'\uplus T''$ then we will also write
 $T=T'\oplus T''$.  Define the {\em zipper}, $Z_n$, to be
 $$
 Z_n = C_n \oplus C_n
 $$
 A picture of $Z_3$ will be found in \Cref{Z_3}.
\begin{thm}
The orbits of $Z_n$ can be partitioned into four sets $\cS$, $\cM$, $\cL$, and $\cG$ (for gigantic).  The properties of the orbits is summarized in the following table:
\bce
    \begin{tabular}{|c|c|c|c|c|}
       \hline & $\cS$ &$\cM$ &  $\cL$ & $\cG$\\ \hline
        $\#\cO$ & $2$& $2^{n+1}-1$& $2^{n+1}$ & $2^{n+2}-2$\\[2pt]
        number of $\cO$ & $2^{2n-1}$& $2^{n+1}-2$&  $1$& $2^n$\\ [2pt]
        $\chi(\cO)$ & $2n+2$ & $2^n(2n+1)$ & $2^n(2n+1) +1 $ & $2^n(4n+3)-n-1$\\[2pt]
        $\chih(\cO)$ & $6n+4$& $3\cdot2^n(2n-1)+5$& $3\cdot2^n(2n-1)+5$ & $2^{n-1}(51n-25)+3$\\[2pt]\hline
    \end{tabular}
\ece
Thus $\chi$ and $\chih$ are homometric on $Z_n$.
\end{thm}
\begin{proof}
As usual, we will just give details about the orbit structure.  Since $Z_n\setm\{\zh\}$ is a disjoint union of two copies of $C_n$, we  use Theorems~\ref{C_n:thm} and~\ref{two:thm}.
Let $\cS'$ and $\cL'$ refer to the orbit partition of $C_n$ and use unprimed notation for $Z_n$.

Combining two orbits from $\cS'$ gives $\gcd(2,2)=2$ orbits of $Z_n$ of length $\lcm(2,2)=2$.  Since $\#\cS'=2^{n-1}$, the total number of orbits formed in this way is
$$
2\cdot 2^{n-1}\cdot 2^{n-1}= 2^{2n-1}.
$$
These are the orbits of $\cS$.

Putting together an orbit from $\cS'$ with the unique orbit in $\cL'$ results in $\gcd(2,2^{n+1}-1)=1$ orbit of size
$\lcm(2,2^{n+1}-1)=2^{n+2}-2$.  Now the total number of orbits is
$$
2\cdot 2^{n-1}\cdot 1 =2^n
$$
and they are the orbits in $\cG$.

Finally, the combination of the orbit in $\cL'$ with itself gives $\gcd(2^{n+1}-1,2^{n+1}-1)=2^{n+1}-1$ orbits.  All of these orbits will have length 
$\lcm(2^{n+1}-1,2^{n+1}-1)=2^{n+1}-1$ except for the one containing $\zh$ which will have one more element.  These orbits are precisely the ones in $\cM\uplus\cL$, and so we are done.
\end{proof}


\section{Comments and open questions}
\label{coq}

\begin{figure}
 \begin{tikzpicture}
\fill(2,0) circle(.1);
\fill(0,1) circle(.1);
\fill(4,1) circle(.1);
\fill(-1,2) circle(.1);
\fill(1,2) circle(.1); 
\fill(3,2) circle(.1);
\fill(5,2) circle(.1);
\fill(-1.5,3) circle(.1);
\fill(-.5,3) circle(.1); 
\fill(.5,3) circle(.1);
\fill(1.5,3) circle(.1); 
\fill(2.5,3) circle(.1);
\fill(3.5,3) circle(.1);
\fill(4.5,3) circle(.1); 
\fill(5.5,3) circle(.1);
\draw(-.5,3.5) node{$x$};
\draw(1.5,3.5) node{$y$};
\draw(4,.5) node{$z$};
\draw (-.5,3)--(-1,2)--(-1.5,3)
(.5,3)--(1,2)--(1.5,3)
(2.5,3)--(3,2)--(3.5,3)
(4.5,3)--(5,2)--(5.5,3)
(-1,2)--(0,1)--(2,0)--(4,1)--(5,2) (0,1)--(1,2) (4,1)--(3,2);
\end{tikzpicture}
    \caption{A complete binary tree}
    \label{bin}
\end{figure}
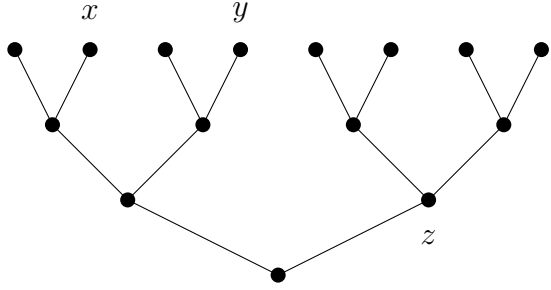

\subsection{Other trees}

The trees considered in the previous section had such nice homometry properties that one might ask if the same is true for other binary trees.  In particular, one could consider the complete binary trees which are those all of whose leaves are at the same rank.  Such a tree is displayed in Figure~\ref{bin}.  Unfortunately, homometry fails for this example tree.  Consider the orbit $\cO$ which contains the antichain $\{x,y\}$ as well as the one $\cO'$ which contains $\{z\}$.  Then it is easy to verify that $\#\cO=\#\cO'=4$.  But
$$
\text{$\chi(\cO)=15\neq 14 =\chi(\cO')$ and $\chih(\cO)=35\neq 26 =\chih(\cO')$.
}
$$

As mentioned in the introduction, Elizalde et al.~\cite{EPRS:rf} considered fences whose Hasse diagrams are paths with any number of minimal elements.  Here we have concentrated on arbitrary trees, but insisted that there be a unique minimal 
element.  It would be interesting to study  trees which are not paths and also not rooted.

\subsection{Piecewise-linear and birational rowmotion}

There are two generalizations of rowmotion which have also been studied and which have consequences for trees.  We first need to describe rowmotion in terms of toggles.   In the discussion which follows  we will just write ideal for lower order ideal. 

If $(P,\le_P)$ is a finite poset and $x\in P$ then the corresponding {\em toggle map} is 
$t_x:\cL(P)\ra \cL(P)$ defined by
$$
t_x(L) = \case{ L\triangle \{x\} }{if $L\triangle \{x\}\in\cL(P)$, }{L}{else}
$$
where $\triangle$ denotes symmetric difference of sets.
A {\em linear extension} of $P$ is a listing of $P$'s elements $x_1,x_2,\ldots,x_p$ such that $x_i\le_P x_j$  implies $x_i$ is weakly left of $x_j$ in the sequence, that is, $i\le j$.  Cameron and Fon-Der-Flaass showed that rowmotion on ideals can be broken into a sequence of toggles.  In what follows we compose functions right to left.
\begin{thm}[\cite{CFDF:oar}]
For any finite poset $P$ and any linear extension
$x_1,x_2,\ldots,x_p$ of $P$ we have

\vs{5pt}

\eqqed{
\rhoh = t_{x_1} t_{x_2}\cdots t_{x_p}.
}
\end{thm}

Stanley~\cite{sta:tpp} introduced the order polytope as a way to use geometry to study posets.
Poset $P=\{x_1,\ldots,x_p\}$ has {\em order polytope}
$$
\Pi(P)=
\{ (f(x_1),\ldots, f(x_p))\in[0,1]^p \mid
\text{$x_i\le_P x_j$ implies $f(x)\le f(y)$}\}.
$$
So $\Pi(P)$ is a subpolytope of the $p$-dimensional unit cube.  Also note that every ideal $L$ of $P$ has a corresponding point of $\Pi(P)$ defined by the function
$$
f(x) =\case{0}{if $x\not\in P$,}{1}{if $x\in P$.}
$$
Einstein and Propp~\cite{EP:plb} extended rowmotion to $\Pi(P)$.  
Write $x\lessdot y$  if $x$ is covered by $y$ in $P$, that is $x<_P y$ and there is no $z$ with 
$x<_P z<_P y$.
If $f\in\Pi(P)$ and $x\in P$ then define the {\em piecewise-linear toggle} $\si_x$ of $f$ at $x$ to be $g=\si_x f\in\Pi(P)$ where
\begin{equation}
\label{rho_PL}
   g(v)=\case{M+m-f(x)}{if $v=x$,}{f(v)}{if $v\neq x$} 
\end{equation}
using the notation
\begin{equation}
   \label{max:min} 
   \text{$M=\max_{y\lec x} f(y)$ and $m=\min_{z\grc x} f(z) $}.
\end{equation}
It is not hard to verify from the definitions that $g\in\Pi(P)$.  One can also show that $\si_x$ is  an involution just like $t_x$, and $\si_x$ is also piecewise-linear as a function.
Finally, one defines {\em piecewise-linear rowmotion},
$\rho_{\PL}:\Pi(P)\ra\Pi(P)$, by
$$
\rho_{\PL} = \si_{x_1}\si_{x_2}\cdots\si_{x_p}
$$
where $x_1,x_2,\ldots,x_p$ is a linear extension of $P$.  It is true, but not obvious from the equation just given, that $\rho_{\PL}$ is well defined in that it does not depend on the chosen linear extension.  Since $\Pi(P)$ has an infinite number of points, it is very possible for orbits of $\rho_{\PL}$ to be infinite. 
However, in certain cases the orbits are nice.  Take, for example, the poset $[p]\times[q]$ which is the poset product of a $p$-element chain and a $q$-element chain.
\begin{thm}[\cite{EP:plb}]
The order of $\rho_{\PL}$ on $[p]\times[q]$ is $p+q$.\hqed
\end{thm}

One can extend piecewise-linear rowmotion even further to the birational realm by detropicalizing as done by Grinberg and Roby~\cite{GR:ipb1,GR:ipb2}. This means that in equations~\eqref{rho_PL} and~\eqref{max:min} sum becomes product, difference becomes quotient, and maximum become sum.  To take care of the minimum, we use the previous dictionary and the fact that for any set $S$ of real numbers
$\min S = -\max(-S)$ where $-S=\{-s \mid s\in S\}$.
Now let $P$ be a finite poset and let $\Ph$ be $P$ with a minimum element $\zh$ and a maximum element $\oh$ added.  Let $\bbF$ be a field and consider a function $f:\Ph\ra\bbF$.  The
{\em birational toggle} of $f$ at $x\in P$ is $g=T_x f$ where
$$
g(v)=
\case{ \dil\frac{\sum_{y\lec x} f(y)}{f(x) \sum_{z\grc x} f(z)^{-1} }}{if $v=x$,}
{f(v)}{if $v\neq x$.\rule{0pt}{20pt}} 
$$
One can verify that $T_x$ is an involution, is a birational function, and that the following is well defined.
Define {\em birational rowmotion} on functions $f:\Ph\ra\bbF$ as
$$
\rho_{\B} = T_{x_1}T_{x_2} \cdots T_{x_p}
$$
where, as usual, $x_1,x_2,\ldots,x_p$ is a linear extension of $P$.  It is even more surprising when birational orbits are finite.  Indeed, $\rho_{\B}$ being of finite order implies this is true for $\rho_{\PL}$.  Again, everything works well for rectangular posets
\begin{thm}[\cite{GR:ipb2}]
The order of $\rho_{\B}$ on $[p]\times[q]$ is $p+q$.\hqed
\end{thm}

Call a poset $P$ {\em graded} if all chains from a minimal element of $P$ to a maximal element have the same length.  Grinberg and Roby consider a class of inductively defined posets which they call skeletal and includes graded rooted forests, that is, disjoint unions of rooted trees such that all leaves have the same rank.  In this context, they prove the following result.
\begin{thm}[\cite{GR:ipb1}]
If $P$ is a skeletal poset then $\rho_{\B}$ has finite order.
\end{thm}
They also give a formula for order of $\rho_B$ in the case that $P$ is a graded rooted forest which agrees with the results in \Cref{estar} for graded extended stars.  A natural question is whether $\rho_{\B}$ has finite order for any rooted trees which are not graded.  Computer experiments suggest that this is not the case, although we have not been able to provide a proof.  Specifically, $200$ trials were run on $16$ posets, and in all but one case the orbit had not repeated after $1,000,000$ iterations of rowmotion.



\nocite{*}
\bibliographystyle{alpha}

\end{document}